\documentclass[11pt]{article}

\usepackage[utf8]{inputenc}
\usepackage[T1]{fontenc}

\usepackage{amsmath,amsfonts,amssymb,latexsym,amsthm,dsfont,amsbsy,stmaryrd}
\usepackage{tikz}

\usepackage{xcolor,graphicx}

\usepackage{enumitem}

\definecolor{myblue}{RGB}{0,51,157}
\usepackage{hyperref}
\hypersetup{colorlinks=true,linkcolor=myblue,citecolor=blue!50,urlcolor=blue} 
\usepackage{url}

\usepackage{fancyhdr}

\usepackage{geometry}

\renewcommand{\theequation}{\arabic{section}.\arabic{equation}}
\let \ssection=\section
\renewcommand{\section}{\setcounter{equation}{0}\ssection}

\newcommand{\ds}[1]{{\displaystyle{#1 }}}

\renewcommand{\leq}{\leqslant}
\renewcommand{\geq}{\geqslant}
\newcommand{\sgn}{\mathrm{sgn}}

\newtheorem{thm}{Theorem}[section]%

\newtheorem{prop}[thm]{Proposition}

\newtheorem{defi}[thm]{Definition}
\newtheorem{rk}[thm]{Remark}
\newtheorem{exple}[thm]{Example}

\newcommand{\ind}{\mathds{1}}
\newcommand{\dd}{{\mathrm d}}

\newcommand{\dN}{\mathbb{N}}
\newcommand{\dP}{\mathbb{P}}
\newcommand{\dR}{\mathbb{R}}

\newcommand{\cA}{\mathcal{A}}
\newcommand{\cB}{\mathcal{B}}
\newcommand{\cC}{\mathcal{C}}

\newcommand{\cL}{\mathcal{L}}

\newcommand{\cQ}{\mathcal{Q}}
\newcommand{\cS}{\mathcal{S}}

\newcommand{\bX}{\mathbf{X}}
\newcommand{\bY}{\mathbf{Y}}
\newcommand{\bZ}{\mathbf{Z}}

\newcommand{\bn}{\mathbf{n}}
\newcommand{\bx}{\mathbf{x}}
\newcommand{\by}{\mathbf{y}}
\newcommand{\bz}{\mathbf{z}}

\newcommand{\rd}{\mathrm{d}}

\newcommand{\ABS}[1]{{{\left| #1 \right|}}} 
\newcommand{\BRA}[1]{{{\left\{#1\right\}}}} 
\newcommand{\NRM}[1]{{{\left\| #1\right\|}}} 
\newcommand{\PAR}[1]{{{\left(#1\right)}}} 
\newcommand{\SBRA}[1]{{{\left[#1\right]}}} 
\newcommand{\DSBRA}[1]{{{\llbracket#1\rrbracket}}} 

\title{Well-posedness of state-dependent rank-based interacting systems}

\author{H\'el\`ene Gu\'erin \footnote{Département de Mathématiques, Université du Québec à Montréal (UQAM), Montréal, QC H2X 3Y7, Canada; guerin.helene@uqam.ca}
\and
Nathalie Krell \footnote{Univ Rennes, IRMAR – UMR CNRS 6625, F-35000 Rennes, France; nathalie.krell@univ-rennes2.fr} 
}

\begin{document}

\maketitle

\begin{abstract}
We study the existence and uniqueness of rank-based interacting systems of stochastic differential equations. These systems can be seen as modifications with state-dependent coefficients of the Atlas model in mathematical finance. The coefficients of the underlying SDEs are possibly discontinuous. We first establish strong well-posedness for a planar system with rank-dependent drift coefficients, and non-rank-dependent and non-uniformly elliptic diffusion coefficients. We then state weak well-posedness for two classes of high-dimensional rank-based interacting SDEs with elliptic diffusion coefficients. Finally, we address the positivity of solutions in the case where the diffusion coefficients vanish at zero.
\end{abstract}

\noindent \textit{MSC2020 subject classifications:} 60H10, 60J60, 60K35. 


\section{Introduction}
We focus in this article on finite systems of diffusive interacting particles driven by stochastic differential equations with coefficients depending on the rank of each particle in the population. A famous such system of stochastic differential equations (SDEs), called the Atlas model, has been introduced by Fernholz, \cite {Fer02}, to model equity markets in finance. Since then many questions related to this model have been studied, such as strong well-posedness \cite{FIKP13,IKS13}, propagation of chaos \cite{JR13, Rey17,KS18}, infinite size case \cite{banerjee2025}, invariant measure  \cite{PalPitman2008, IKS13,tsai2018}, large deviation principle \cite{dembo16}, to name a few. This stochastic system is constructed from two families of constants to formulate the drift and diffusion coefficients, and a vector of independent Brownian motions. In this work, we aim to take the first steps toward extending the classic Atlas model to SDEs with coefficients that also depend on the processes.

To motivate the need to generalize the  Atlas model, one can mention the presentation  of Mahé and Bugeon, \cite{bugeon-mahe18}, and  \cite{MullerFeuga1990, Katsanevakis2006, Dampin2012} on  fish growth in a fish pond. Let us assume that the population is of constant size $N\geq 2$ in the fish pond.
We can easily imagine that dominant fish will grow faster than dominated fish, since food is distributed in limited quantities. It is thus natural to consider a system of rank-based interacting diffusion processes to model the fish growth: the coefficients of the stochastic differential equation satisfied by the weight of each fish will depend on its rank in the population. However, in such a model, it is unrealistic to have a family of drift and diffusion coefficients that does not depend on fish weights, as is the case in the Atlas model.

A natural extension of the Atlas model is the following. For $t\geq 0$, let $\bX_t=(X^{1}_t,\ldots,X^{N}_t)\in\dR^N$, where $X^{i}_t$ can be seen, for example, as the weight of the $i^\text{th}$ fish at time $t$ in the population of a fish pond.
We assume that the behavior of the $i^\text{th}$ coordinate satisfies the following stochastic differential equation 
\begin{align}\label{eq:EDS-poisson}
    \dd X^{i}_t&=\sum_{k=1}^Nb_k(X^{i}_t)\ind_{X^{i}_t=X^{(k)}_t}\dd t
+\sum_{k=1}^N\sigma_k(X^{i}_t)\ind_{X^{i}_t=X^{(k)}_t}\dd W^i_t,
\end{align}
where 
$\PAR{W^i}_{1\leq i\leq N}$ are $N$ independent Brownian motions, and $X^{(1)},X^{(2)},\ldots,X^{(N)}$ are the order statistics of the N-tuple $(X^{1},X^{2},\ldots,X^{N})$: for all $t\geq 0$,
\[
\min_{1\leq i\leq N}X^{i,N}_t=X^{(1)}_t\leq X^{(2)}_t\leq \ldots\leq X^{(N)}_t=\max_{1\leq i\leq N}X^{i,N}_t,
\]
with initial condition
$\bX_0=(X_0^{1},\ldots, X_0^{N})\in \mathrm{L}^2(\dR^N)$ with $X_0^{1}<X_0^{2}<\ldots<X_0^{N}$.

\medskip
The corresponding $N$-dimensional SDE linked to ~\eqref{eq:EDS-poisson} satisfied by $\bX$ has discontinuous drift and diffusion coefficients. Even when the families $\PAR{b_1,\ldots,b_N}$ and $(\sigma_1,\ldots,\sigma_N)$ are smooth bounded functions on $\dR^N$, the well-posedness of this system is not clear. 
In the classical Atlas model, these families are constant functions. Weak existence and uniqueness of a solution to the Atlas model can be deduced from Bass and Pardoux's article \cite{BP87} on diffusions with piecewise constant coefficients, as mentioned in \cite{BFK05, FK09}. In a sequence of articles, the strong existence has been studied, and obtained first for $N=2$ in \cite{FIKP13}, then for any $N\geq 3$ in \cite{IKS13}. In the case $N\geq 3$, Ichiba et al. obtained the strong existence up to the first triple collision time of the system. They also identified a condition on the diffusion coefficients to avoid such a triple collision. 
A more general setting, such as Equation~\eqref{eq:EDS-poisson}, is difficult to tackle. We can mention the work of Itkin and Larsson \cite{ItkinLarsson21}, but an assumption of their main result is the continuity of the coefficients, which is not the case in the Atlas model and its generalization~\eqref{eq:EDS-poisson}.

The system~\eqref{eq:EDS-poisson} can be simplified in different ways, and, for example, we can consider the following rank-based interaction system, in which the diffusion coefficients do not depend on the rank: 
\begin{align}\label{eq:EDS-poisson-2}
    \dd X^{i}_t&=\sum_{k=1}^Nb_k(X^{i}_t)\ind_{X^{i}_t=X^{(k)}_t}\dd t
+\sigma_i(X^{i}_t)\dd W^i_t.
\end{align}

The stochastic systems~ \eqref{eq:EDS-poisson} and ~\eqref{eq:EDS-poisson-2} belong to the class of SDEs with discontinuous drifts. There is a wide literature on the well-posedness of this class of SDEs; without claiming to be exhaustive, we can mention \cite{Veretennikov81,IW89,gao93,krylov04,SV06}, and the recent article \cite{Krylov21}. Finally, we mention the different works \cite{LTS15,Leobacher-Szolgyenyi2017,Leobacher-Szolgyenyi2019-corr} on the strong existence of SDEs with discontinuous drifts. In all of these articles, uniform ellipticity of the diffusion term is crucial in the proofs.

\medskip
However, recall that our primary motivation was modeling the weights of  fish in a fish pond. As such quantities are nonnegative, we need to ensure that the process $\bX$ stays in the nonnegative cone $\dR_+^N$. One way to ensure this is to assume that the diffusion coefficients cancel out as soon as a coordinate reaches zero (see Section~\ref{App:nonnegativity}), as in a stochastic logistic model \cite{Braumann19} or Cox–Ingersoll–Ross 
 model \cite{CIR85,heston93}. Consequently, we want to avoid the uniform ellipticity assumption on the diffusion term for the well-posedness of the equation. The general case $N\geq 2$ being difficult to tackle, we focus on $N=2$, as it was first studied to obtain the strong existence and uniqueness of the Atlas model.

\medskip 
The main contribution of this article is to prove the strong existence and uniqueness of the solution to \eqref{eq:EDS-poisson-2} for $N=2$ for a general class of drift and diffusion coefficients that include logistic-type models.
The proof, presented in Section~\ref{sec:main result}, is greatly inspired by \cite{Leobacher-Szolgyenyi2017}, although our coefficients do not satisfy their assumptions, including the uniform ellipticity of the diffusive term. Consequently, we had to adapt their approach to fit our assumptions.

In Section~\ref{sec: weak result}, we present concise proofs of weak existence and uniqueness in any dimension $N\geq 2$: first for the model~\eqref{eq:EDS-poisson} when the diffusive coefficients are positive and not state-dependent, then for the second model~\eqref{eq:EDS-poisson-2} under the uniform ellipticity condition. At the end of that section, we further address the non-negativity of the solutions of \eqref{eq:EDS-poisson} and \eqref{eq:EDS-poisson-2} when the diffusion coefficients vanish at zero. 

\bigskip
{\bf Notations}
\begin{itemize}
\item $\NRM{\cdot}$ denotes the Euclidean distance;
    \item For $x,y\in\dR$, $x\vee y=\max(x,y)$ and $x\wedge  y=\min(x,y)$.
\end{itemize}

\section{Strong well-posedness of a planar rank-based interacting system}\label{sec:main result}

In this section, we focus on the system~\eqref{eq:EDS-poisson-2} with $N=2$. Consequently, we consider two pairs of measurable functions, $(b_1,b_2)$ and $(\sigma_1,\sigma_2)$, defined on $\dR$ with values in $\dR^2$. The system~\eqref{eq:EDS-poisson-2} satisfied by $\bX=(X^1,X^2)$ can thus be written  
\begin{equation}\label{eq:EDSX}
\dd \bX_t=b(\bX_t)\dd t+\sigma(\bX_t)\dd W_t,
\end{equation}
with $W=(W^1,W^2)$ a 2-dimensional Brownian motion, and for $\bx= \begin{pmatrix} x_1\\ x_2 \end{pmatrix}$
\begin{equation}\label{eq:def-b}
b(\bx)=\begin{pmatrix}b^1(\bx)\\
b^2(\bx)
\end{pmatrix}:=
\begin{pmatrix}
b_1(x_1)\ind_{x_1\leq x_2}+b_2(x_1)\ind_{x_1>x_2}\\
b_1(x_2)\ind_{x_2<x_1}+b_2(x_2)\ind_{x_2\geq x_1}
\end{pmatrix},
\end{equation}
and
\[
\sigma(\bx):=\begin{pmatrix}
\sigma_1(x_1)&0\\
0&\sigma_2(x_2)
\end{pmatrix}.
\]

We assume that the coefficients $(b,\sigma)$ satisfy the following assumption:
\begin{quote}
{\bf Assumption $\PAR{A_{b,\sigma}}$}:\\
$b_1$, $b_2$, and $\sigma$ are locally Lipschitz-continuous functions, and
$\sigma$ is nonnegative on $\dR^2$.
\end{quote}

Let $\Theta=\BRA{\bx\in\dR^2:x_1=x_2}$. For $c>0$, we denote by $\Theta_c=\BRA{\bx\in \dR^2: \rd (\bx,\Theta)<c}$, where $\rd (\bx,\Theta)=\inf\BRA{\NRM{\bx-\by}:\by\in\Theta}$.

\medskip
Our aim is to prove the strong existence and uniqueness of the process $\bX$. Under assumption $(A_{b,\sigma})$, the coefficients of the SDE~\eqref{eq:EDSX} are \emph{locally Lipschitz-continuous} on the set $\dR^2\setminus\Theta$ and then satisfy the usual assumptions to ensure strong existence and uniqueness, so the only critical region is $\Theta$. By adapting the arguments of Leobacher and Szolgyenyi, \cite{Leobacher-Szolgyenyi2017}, we cleverly distort the process $\bX$ when it enters a neighborhood $\Theta_c$ of $\Theta$, for $c$ sufficiently small. 
To this end, we will introduce an invertible function $G:\dR^2\to\dR^2$ of class $\cC^1$ such that the new process $\bZ:=G(\bX)$ satisfies an SDE with locally Lipschitz-continuous coefficients on $\dR^2$. 

Let us mention that the process $\bX$ defined by \eqref{eq:EDSX} can be seen as a particular case of the article \cite{Leobacher-Szolgyenyi2017}. However, our goal is to relax some of their assumptions, in particular to extend their result to our model for non-uniformly elliptic diffusion coefficients.
In addition, the proof is simpler in the current model than in the general framework studied in \cite{Leobacher-Szolgyenyi2017}, and gives hope to extend the result to a higher dimension $N\geq 2$.

\medskip
We begin by providing some necessary definitions.

\begin{defi}\cite[Definition 3.1]{Leobacher-Szolgyenyi2017}
Let $A \subseteq \mathbb{R}^{d}$ with $d\geq 1$. The intrinsic metric $\rho$ on $A$ is given by
\[
\rho(x, y):=  \inf \BRA{\ell(\gamma)\text{ with } \gamma:[0,1] \rightarrow A \text { a continuous curve  satisfying } \gamma(0)=x, \gamma(1)=y}
\]

and $\rho(x, y):=\infty$ when there is no continuous curve from $x$ to $y$.\\

\end{defi}
\begin{defi} \cite[Definition 3.2]{Leobacher-Szolgyenyi2017} Let  $d,m\geq 1$, and $A \subseteq \mathbb{R}^{d}$. Let $f: A \longrightarrow \mathbb{R}^{m}$ be a function. We say that $f$ is intrinsic Lipschitz if it is Lipschitz w.r.t. the intrinsic metric on $A$, that is, if there exists a constant $L>0$ such that
\[
\forall x, y \in A: \quad\|f(x)-f(y)\| \leq L \rho(x, y).
\]
\end{defi}

\begin{defi}\cite[Definition 3.4]{Leobacher-Szolgyenyi2017} Let $I \subseteq \mathbb{R}$ be an interval. We say a function $f: I \longrightarrow \mathbb{R}$ is piecewise Lipschitz if there are finitely many points $\xi_{1}<\cdots<\xi_{m} \in I$ such that $f$ is Lipschitz on each of the intervals $\left(-\infty, \xi_{1}\right) \cap I,\left(\xi_{m}, \infty\right) \cap I$ and $\left(\xi_{k}, \xi_{k+1}\right)$, $k=1, \ldots, m$.
\end{defi}

\begin{rk} Note that for a function $f: \mathbb{R} \longrightarrow \mathbb{R}$, we have that $f$ is piecewise Lipschitz, if and only if $f$ is intrinsic Lipschitz on $\mathbb{R} \backslash B$, where $B$ is a finite subset of $\mathbb{R}$.
\end{rk}
\medskip

Under Assumption~$\PAR{A_{b,\sigma}}$, we easily notice that $b$ defined by \eqref{eq:def-b} is a piecewise Lipschitz function $\dR^2\to\dR^2$, with  exceptional set $\Theta$.

For $\bx\in\Theta$, we consider the orthonormal vector $\bn=\frac{1}{\sqrt{2}}\begin{pmatrix}
1\\
-1
\end{pmatrix}$ of $\Theta$ at $\bx$, and for $\bx\in\dR^2$ we denote by $p(\bx)$ the projection of $\bx$ onto $\Theta$. We have

\[p(\bx)=\begin{pmatrix}
\frac{x_1 + x_2}{2}\\
\frac{x_1 + x_2}{2}
\end{pmatrix}.\]

We introduce the following function of class $\cC^3$ with compact support on $\dR$: for $u\in\dR$,
\begin{equation}\label{eq: def phi}
\phi(u)=\PAR{1-u^2}^4\ind_{\ABS{u}\leq 1}.
\end{equation}

For a fixed $c>0$ (which will be chosen later), we introduce  the function $\widetilde \phi:\dR^2\to\dR$ by
\begin{equation}\label{eq:def phi tilde}
\widetilde{\phi}(\bx)= \phi \PAR{\frac{x_1 -x_2 }{\sqrt{2} c}} \alpha\PAR{\frac{x_1+x_2}{2}}
\end{equation}
with, for $u\in\dR$, 
\begin{equation}\label{eq:def-alpha}
\alpha (u)=\frac{1}{\sqrt{2}}\dfrac{b_1 (u )-b_2 (u )}{\sigma_1 (u)^2 + \sigma_2 (u)^2 }.
\end{equation}
We notice that $\frac{\ABS{x_1 -x_2} }{\sqrt{2} c}=\frac{\NRM{\bx-p(\bx)}}{c}=\frac{\rd(\bx,\Theta)}{c}$, hence $\widetilde{\phi}\equiv 0$ on $\dR^2\setminus\Theta_c$. 

We now introduce the function $G$:
\begin{align}\label{eq:def G}
G (\bx )
&=\bx+ (x_1 -x_2 )|x_1 -x_2 |\widetilde{\phi}(\bx) 
\bn \notag\\
&=\bx+\frac{1}{\sqrt{2}}(x_1 -x_2 )^2\sgn(x_1 -x_2) \widetilde{\phi}(\bx) \begin{pmatrix}1\\-1\end{pmatrix},
\end{align}
where $\sgn$ is the  sign function: $\sgn(x)=1$ for $x>0$, $\sgn(x)=-1$ for $x<0$, and $\sgn(x)$ for $x=0$.
We observe that $G(\bx)=\bx$ for $\bx\notin \Theta_c$ and  $G(\bx)=\bx$ for $\bx\in \Theta$.

\medskip
We now make some assumptions on the function $\alpha$.\\
{\bf Assumption $\PAR{A_\alpha}$}
\begin{enumerate}
\item   $\alpha$ is of class $\cC^2(\dR)$,
\item  $\alpha$ is bounded on $\dR$.
\end{enumerate}

Our main result for $N=2$ is the following.
\begin{thm}\label{thm:main result} We suppose that Assumptions $\PAR{A_{b,\sigma}}$ and $\PAR{A_\alpha}$ are satisfied.

Let $\bX_0$ be a random variable independent of $W$.
    There is a unique strong solution $\bX$ to the SDE \eqref{eq:EDSX} with initial condition $\bX_0$, up to a potential explosion time $e(X)\in(0,\infty]$.
\end{thm}

Note that, when $b_1,b_2,\sigma$ are Lipschitz-continuous with $b_1,b_2$ bounded, there is no explosion of the solution, i.e. $e(X)=\infty$ a.s. (see the comment at the beginning of the Step 3 in the proof of Theorem~\ref{thm:SDE-Z}).

The proof of Theorem~\ref{thm:main result} is given in Section~\ref{sec:proof main thm}.
Before diving into the theorem's proof, let us make a few comments on Assumption $\PAR{A_\alpha}$ and on the construction of the function $G$, and give numerical simulations for some examples.
\begin{rk}\label{rk:comment on the planar well posedness}
\begin{enumerate}
\item
The first assumption of $\PAR{A_\alpha}$ will be useful to prove that $G$ is of class $\cC^1$ on $\dR^2$. 
\item The second assumption of $\PAR{A_\alpha}$ will allow us to prove that $G$ is invertible. We note that when $b_1$ and $b_2$ are bounded, and $\sigma_1$, $\sigma_2$ are bounded below by a strictly positive constant, then $\alpha$ is bounded. This is a sufficient, but not necessary, condition for $\alpha$ to be bounded. Indeed, under good assumptions on the functions, we can assume that $\sigma_1$ and $\sigma_2$ are only nonnegative.  When the roots of $\sigma_1$ and $\sigma_2$ are distinct, then the function $\sigma_1^2+\sigma_2^2$ is always positive. When $\sigma_1$ and $\sigma_2$ have a common root, we then need to assume that $b_1$ and $b_2$ are equal at this point, and such that $\alpha$ is well defined and bounded. We can think for example of Logistic-type models, see Example~\ref{exple:logistic model} below. 
\item The main difference with the work of \cite{Leobacher-Szolgyenyi2017} is in the choice of the function $\alpha$, which in our case allows us to consider non uniformly elliptic diffusion coefficients, as mentioned above.
\item The construction of the function $G$ cannot be directly  extended to higher dimensions. In dimension $N\geq3$, the set $\Theta=\BRA{\bx\in\dR^N:\exists i\neq j 
\text{ with }x_i=x_j}$ is the union of hyperplanes.
The orthogonal projection of a point $\bx\in\dR^N$ on $\Theta$ and the orthogonal direction $\bn$ to $\Theta$ are not always unique, and thus well defined. Therefore, the way in which the function $G$ is constructed cannot guarantee its regularity.
\end{enumerate}
\end{rk}

Let us note that Leobacher and Sz\"olgyenyi \cite{Leobacher-Szolgyenyi2017} proved the convergence of the following Euler-Maruyama scheme to simulate the process $\bX$ starting from an initial value $\bx_0$: let $\delta>0$ be fixed, and 
\begin{itemize}
    \item Let $\bz_{0} =G(\bx_0)$;
    \item Apply the classical Euler-Maruyama scheme to SDE (\ref{eq:Z}) to obtain an approximation $\bZ^{(\delta)} $ of $\bZ$ on the time interval $[0,T]$ with step size $\delta$;
    \item Set $\bX^{(\delta)} =G^{-1} (\bZ^{(\delta)} )$ as an approximation of $\bX$ on $[0,T]$.
\end{itemize}
Their proof works with our model under Assumptions $(A_{b,\sigma})$ and $(A_\alpha)$ when the functions $b_k$ are bounded.
However in the following example, we used a classical Euler-Maruyama scheme, since the inverse of the function $G$ is not easy to compute.

\begin{exple}\label{exple:logistic model}
We simulate trajectories for two logistic-type models that satisfy the assumptions of Theorem~\ref{thm:main result} when $N=2$. 
\begin{enumerate}
    \item Logistic-type model I: we consider the coefficients $
 b_k(x)=r_kx^2\PAR{1-\frac{x}{x_\mathrm{max}}}_+$ and $  \sigma(x)=\sigma_0 x$, where $u_+:=\max(u,0)$.
 We easily observe that the coefficients satisfy the assumptions of Theorem~\ref{thm:main result}.
In Figure~\ref{fig:logistic model 1}, we observe the trajectories of the interacting processes when $r_k=0.7+0.1\frac{k}{N}$, $x_{\max}=10$ and $\sigma_0=1/20$.
 \begin{figure}[h!]
     \centering
     \includegraphics[scale=.5]{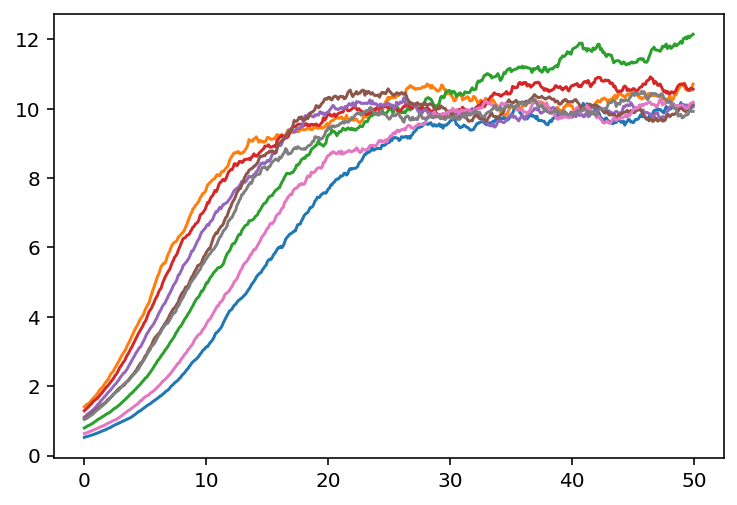}
     \caption{Trajectories of the Logistic-type model I  until time $T=50$ for $N=8$.}
     \label{fig:logistic model 1}
 \end{figure}
 
\item Logistic-type model II: we now consider the coefficients  $b_k(x)=r_kx^2\PAR{1-\frac{x}{x_\mathrm{max}}}^2$, and $\sigma(x)=\sigma_0 x\PAR{1-\frac{x}{x_\mathrm{max}}}$ on $[0,x_{\max}]$. In this second logistic-type model, the fluctuations are canceled out  when the process reaches the maximal value $x_{\max}$ as we can see in Figure~\ref{fig:logistic model 2} with $r_k=0.7+0.1\frac{k}{N}$, $x_{\max}=10$, $\sigma_0=1/8$. This model is more realistic and the shape of the curves fits better with the observations of Mahé and Bugeon \cite{bugeon-mahe18}.

\begin{figure}[h!]
    \centering
    \includegraphics[scale=.5]{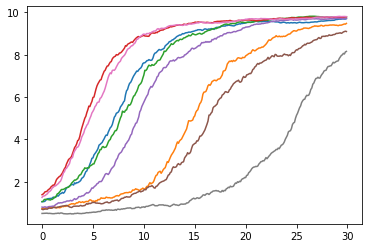}
    \caption{Trajectories of the Logistic-type model II  until time $T=30$ for $N=8$.}
    \label{fig:logistic model 2}
\end{figure}

\end{enumerate}
\end{exple}

The proof of Theorem~\ref{thm:main result}, given in Section~\ref{sec:proof main thm}, requires different intermediate results: first on the properties of the function $G$ and then on the well-posedness of the SDE satisfied by $\bZ=G(\bX)$. 
\subsection{Properties of the function $G$}

 We first prove that $G$ has the required regularity for $c$ small enough.
\begin{prop}\label{prop:G-diffeo}
Under Assumption $\PAR{A_{\alpha}}$, there exists $c_0>0$ such that $\forall c\in(0,c_0)$,  
the function $G$, defined by~\eqref{eq:def G}, is a $\cC^1$-diffeomorphism on $\dR^2$.
\end{prop}

 \begin{proof}
The function $\phi$ is of class $\cC^3$ on $\dR$ and by assumption, $\alpha$ is of class $\cC^2$ on $\dR$. Thus the function $\widetilde{\phi}$, defined by~\eqref{eq:def phi tilde}, is of class $\cC^2$ on $\dR^2$. We also easily observe that $G$ is continuous on $\dR^2$.

Let us compute the first derivatives of $G$. We have for $x\in\dR^2\setminus\Theta$,
\begin{align}
\frac{\partial G}{\partial x_1}(\bx)&=\begin{pmatrix}1\\0\end{pmatrix}+\frac{2}{\sqrt{2}}(x_1 -x_2 )\sgn(x_1 -x_2 )\widetilde{\phi}(\bx) \begin{pmatrix}1\\-1\end{pmatrix}
\nonumber\\
&\hskip 1cm +\frac{1}{\sqrt{2}}(x_1 -x_2 )^2\sgn(x_1 -x_2)\partial_{x_1}\widetilde{\phi}(\bx) \begin{pmatrix}1\\-1\end{pmatrix}
\label{eq:G'1}\\
\frac{\partial G}{\partial x_2}(\bx)&=\begin{pmatrix}0\\1\end{pmatrix}-\frac{2}{\sqrt{2}}(x_1 -x_2 )\sgn(x_1 -x_2 )\widetilde{\phi}(\bx) \begin{pmatrix}1\\-1\end{pmatrix}
\nonumber\\
&\hskip 1cm +\frac{1}{\sqrt{2}}(x_1 -x_2 )^2\sgn(x_1 -x_2)\partial_{x_2}\widetilde{\phi}(\bx) \begin{pmatrix}1\\-1\end{pmatrix}
\label{eq:G'2}
\end{align}
The first derivatives of $G$ being continuous on $\dR^2$, we deduce that $G$ is of class $\cC^1$ on $\dR^2$.

We now prove that $G$ is invertible.
The determinant of the Jacobian matrix $G':=\PAR{\frac{\partial G}{\partial x_1},\frac{\partial G}{\partial x_2}}$ of $G$ is given by
\begin{align*}
\det(G'(\bx))&=1+\frac{4}{\sqrt{2}}\ABS{x_1-x_2}\widetilde\phi(\bx)+\frac{1}{\sqrt{2}}(x_1-x_2)\ABS{x_1-x_2}\PAR{\partial_{x_1}\widetilde{\phi}(\bx)-\partial_{x_2}\widetilde{\phi}(\bx)}\\
&=1+\frac{4}{\sqrt{2}}\ABS{x_1-x_2}\phi\PAR{\frac{x_1-x_2}{\sqrt{2}c}}\alpha\PAR{\frac{x_1+x_2}{2}}\\
&\hskip 3cm+\frac{1}{c}(x_1-x_2)\ABS{x_1-x_2}\phi'\PAR{\frac{x_1-x_2}{\sqrt{2}c}}\alpha\PAR{\frac{x_1+x_2}{2}}.
\end{align*}

We notice that $z\phi'(z)=-8z^2(1-z^2)^3\ind_{\ABS{z}\leq 1}\leq 0$. By assumption, the function $\alpha$ is bounded on $\dR$, and since the support of $\phi$ is $[-1,1]$, for $c$ small enough, i.e. $c<\PAR{16\NRM{\alpha}_\infty}^{-1}$, $\det(G'(\bx))>0$ for all $\bx\in\dR^2$.

Moreover $G$ is of class $\mathcal{C}^1$ and $\lim_{||\bx||\rightarrow\infty } ||G(\bx)||=\infty$ (as $\phi$ is equal to zero outside a compact set and $\alpha$ is bounded), so we can apply Hadamard's global inverse theorem, see Theorem 2.2 in \cite{MR3319979}, to conclude that $G$ is a $\cC^1$-diffeomorphism on $\dR^2$.

\end{proof}

 \subsection{The stochastic differential equation satisfied by $G(\bX)$}

 In this section, we focus on the new process $\bZ=G(\bX)$.
 \begin{thm}\label{thm:SDE-Z}
Suppose that Assumptions $\PAR{A_{b,\sigma}}$ and $\PAR{A_\alpha}$ are satisfied.
 
For $c\in(0,c_0)$, with $c_0$ given in Theorem~\ref{prop:G-diffeo}, the process $\bZ:=G(\bX)$ is the solution of an SDE with locally Lipschitz continuous coefficients on $\dR^2$.
 \end{thm}
 \begin{proof}
The proof is divided into three steps. Following the ideas of \cite{Leobacher-Szolgyenyi2017}, we first prove that Itô's formula holds for $G(\bX)$ and $G^{-1}(\bX)$, then we prove that the drift coefficient of the SDE of $\bZ$ is continuous on $\dR^2$, and finally we deduce that both coefficients are locally Lipschitz continuous on $\dR^2$.

\begin{quote}
\underline{Step 1}: Itô's formula holds for $G$ and $G^{-1}$.
\end{quote}

$G$ and $G^{-1}$ are
of class $\cC^2$ on $\dR^2\setminus\Theta$. Then if $\bx\in\dR^2\setminus\Theta$, Itô's formula holds until $\bX$ hits $\Theta$.
For $\bx\in\Theta$, 
we first remark that the set $\Theta$ can be parametrized by the function $\psi(u)=\begin{pmatrix}u\\u\end{pmatrix}$ for $u\in\dR$. For $\by\in\dR^2$, we define
\begin{equation}\label{eq:tau}
\tau (\by)=y_1 \bn +\psi (y_2 )=\begin{pmatrix}
\frac{y_1}{\sqrt{2}}+y_2\\
-\frac{y_1}{\sqrt{2}}+y_2
\end{pmatrix}
\end{equation}
and its inverse $\cS:=\tau^{-1}$,
\begin{equation}\label{eq:S}
\cS (\bx)=\begin{pmatrix}
\frac{x_1 - x_2}{\sqrt{2}}\\
\frac{x_1 + x_2}{2}
\end{pmatrix}.
\end{equation}

We recall that 
$\Theta_c=\BRA{\bx\in\dR^2:\mathrm{d}(\bx,\Theta)<c}$, then we remark that
\[
\Theta_c=\BRA{\tau(\by):\ABS{y_1}< c,y_2\in\dR}.
\]

The functions $\tau$ and $\cS$ are linear on $\dR^2$, then Itô's formula holds for those functions. We now prove that it also holds for $G\circ \tau$, and consequently it will hold for $G=G\circ \tau\circ \cS$.

We first note that
\begin{align}
\nonumber
G\circ \tau (\by)
&=\begin{pmatrix}
\frac{y_1}{\sqrt{2}}+y_2\\
-\frac{y_1}{\sqrt{2}}+y_2
\end{pmatrix}+ \frac{2}{\sqrt{2}}y_1^2\sgn(y_1)\widetilde{\phi}( \tau (\by)) \begin{pmatrix}
1\\
-1
\end{pmatrix}
\\
&=\begin{pmatrix}
\frac{y_1}{\sqrt{2}}+y_2\\
-\frac{y_1}{\sqrt{2}}+y_2
\end{pmatrix}+ \frac{2}{\sqrt{2}}y_1^2\sgn(y_1) \phi\PAR{\frac{y_1}{c}}\alpha ( y_2)\begin{pmatrix}
1\\
-1
\end{pmatrix}.
\label{eq:Gtau}
\end{align}

We observe that $G\circ \tau$ is a function of class $\cC^1$ on $\dR^2$, with second derivatives on $\dR^*\times\dR$. We also note that $\frac{\partial^2 G\circ \tau}{\partial y_{2}^2}$ and $\frac{\partial^2 G\circ \tau}{\partial y_{1} \partial y_{2} }$ are continuous on $\dR^2$, and $\frac{\partial^2 G\circ \tau}{\partial y_{1}^2}$ is continuous on $\dR^*\times \dR$ with $\sup_{\{\by|y_1 \not=0\}} |\frac{\partial^2 G\circ \tau}{\partial y_{1}^2}(\by)|<\infty$, because $\alpha$ is bounded on $\dR$. 
By \cite[Theorem 2.9]{LTS15},  we deduce that Itô's formula holds for $G\circ \tau$, and then for $G$.

Introducing 
the function $\cS\circ G^{-1}\circ \tau=\PAR{\cS\circ G\circ \tau}^{-1}$, we obtain similarly that Itô's formula also holds for $G^{-1}$.

\begin{quote}
\underline{Step 2}: $\bZ=G(\bX)$ satisfies an SDE with a continuous drift coefficient.
\end{quote}

Recall that $\bX$ satisfies \eqref{eq:EDSX} and $\cS$ is the linear function defined by \eqref{eq:S}, we deduce that $\bY:=\cS(\bX)$ satisfies
\begin{align}
\dd Y^{1}_t &=\frac{1}{\sqrt{2}}\dd X^{1}_t -\frac{1}{\sqrt{2}}\dd X^{2}_t \nonumber\\
&=\frac{1}{\sqrt{2}}\PAR{b^1(\tau(\bY))-b^2(\tau(\bY))}\dd t+\frac{1}{\sqrt{2}}\sigma_1(\tau(\bY))\dd W^1_t-\frac{1}{\sqrt{2}}\sigma_2(\tau(\bY))\dd W^2_t  
\label{eqY1}\\
\dd Y^{2}_t &=\frac{1}{2}\dd X^{1}_t +\frac{1}{2}\dd X^{2}_t \nonumber \\
&=\frac{1}{2}\PAR{b^1(\tau(\bY))+b^2(\tau(\bY))}\dd t+\frac{1}{2}\sigma_1(\tau(\bY))\dd W^1_t+\frac{1}{2}\sigma_2(\tau(\bY))\dd W^2_t\label{eqY2}
\end{align}
with
\[
b(\tau(\by))=\begin{pmatrix}
    b_1\PAR{\frac{y_1}{\sqrt{2}}+y_2}\ind_{y_1<0}+b_2\PAR{\frac{y_1}{\sqrt{2}}+y_2}\ind_{y_1>0}\\
    b_1\PAR{-\frac{y_1}{\sqrt{2}}+y_2}\ind_{y_1>0}+b_2\PAR{-\frac{y_1}{\sqrt{2}}+y_2}\ind_{y_1<0}
\end{pmatrix}.
\]

\vskip 0.2cm
We consider the function $g:= \mathcal{S}\circ G\circ \tau$. As $\tau\circ\cS=id$, we remark that $G(\bX)=\tau\circ\cS\circ G\circ \tau \circ\cS(\bX)=\tau\circ g(\bY)$.
We just need to prove that the drift coefficient of $g(\bY)$ is continuous. As $\tau$ is a linear function, we will then easily deduce  the continuity of the drift coefficient of the SDE for the process $\bZ=G(\bX)$.

\vskip 0.2cm
We have, from \eqref{eq:Gtau},
\begin{align*}
g(\by)&=\begin{pmatrix}
y_1\\
y_2
\end{pmatrix}+2 y_1^2\sgn(y_1) \phi\PAR{\frac{y_1}{c}}\alpha ( y_2)\begin{pmatrix}
1\\
0
\end{pmatrix}.
\end{align*}

Let us compute the successive derivatives of $g$. We have
\begin{align*}
\frac{\partial g}{\partial y_1}
(\by)&=\begin{pmatrix}
1+4y_1\sgn(y_1) \phi(\frac{y_1}{c})\alpha( y_2)+\frac{2}{c}y_1^2\sgn(y_1)\phi'\PAR{\frac{y_1}{c}}\alpha(y_2)\\
0
\end{pmatrix}
\\[0.2cm]
\frac{\partial g}{\partial y_2}
(\by)&=\begin{pmatrix}
2y_1^2\sgn(y_1) \phi(\frac{y_1}{c})\alpha' ( y_2)\\
1
\end{pmatrix}
\\[0.2cm]
\frac{\partial^2 g}{\partial y_{1}^2}
(\by)&=
\begin{pmatrix}
\SBRA{4\sgn(y_1) \phi(\frac{y_1}{c})+\frac{8}{c}y_1\sgn(y_1) \phi'(\frac{y_1}{c})+\frac{2}{c^2}y_1^2\sgn(y_1)\phi"\PAR{\frac{y_1}{c}}}\alpha(y_2)\\
0
\end{pmatrix}
\\[0.2cm]
\frac{\partial^2 g}{\partial y_{2}^2}
(\by)&=\begin{pmatrix}2y_1^2\sgn(y_1) \phi(\frac{y_1}{c})\alpha'' ( y_2)\\
0
\end{pmatrix}
\\[0.2cm]
\frac{\partial^2 g}{\partial y_{1} \partial y_{2} }
(\by)&=\begin{pmatrix}
\SBRA{4y_1\sgn(y_1) \phi(\frac{y_1}{c})+\frac{2}{c}y_1^2\sgn(y_1)\phi'\PAR{\frac{y_1}{c}}}\alpha'(y_2)\\
0
\end{pmatrix}.
\end{align*}

We observe that $g$ is a function of class $\cC^1$ on $\dR^2$, with second derivatives on $\dR^*\times\dR$. We also note that $\frac{\partial^2 g}{\partial y_{2}^2}$ and $\frac{\partial^2 g}{\partial y_{1} \partial y_{2} }$ are continuous on $\dR^2$, and $\frac{\partial^2 g}{\partial y_{1}^2}$ is continuous on $\dR^*\times \dR$ with $\sup_{\{\by|y_1 \not=0\}} |\frac{\partial^2 g}{\partial y_{1}^2}(\by)|<\infty$ ($\alpha$ being bounded). 
By \cite[Theorem 2.9]{LTS15},  we deduce that for all $t \geq 0$

\begin{equation}\label{eq:ItoY}
g(\bY_t)=g(\bY_0)+\sum_{i=1}^2\int_{0}^t \frac{\partial g}{\partial y_i}dY_{s}^i +\frac{1}{2} \sum_{i,j=1}^2\int_{0}^t \frac{\partial^2 g}{\partial y_i\partial y_j}d[Y^i ,Y^j]_s 
\end{equation}

Using the system \eqref{eqY1}-\eqref{eqY2} of SDE satisfied by $\bY$, we observe  the "potential" discontinuous terms in the drift of $g(\bY)$ come from $\frac{\partial g}{\partial y_1}(\by)
\PAR{b^1(\tau(\by))-b^2(\tau(\by))}$  and \\
$\frac{\partial^2 g}{\partial y_1^2}(\by)\PAR{\sigma_1(\tau(\by))^2+\sigma_2(\tau(\by))^2}$. 
Focusing on the first term of $\frac{\partial g}{\partial y_1}$ and $\frac{\partial^2 g}{\partial y_1^2}$ in Itô's formula \eqref{eq:ItoY}, the drift term of $g(\bY)$ is of the form $\nu(\by)=\begin{pmatrix}\nu_1(\by)\\\nu_2(\by)\end{pmatrix}$ with
\begin{align*}
\nu_1(\by)&=
\frac{1}{\sqrt{2}}\PAR{b^1\circ\tau(\by)-b^2\circ\tau(\by)}\\
&+\sgn(y_1)\phi\PAR{\frac{y_1}{c}}\alpha(y_2)\PAR{\PAR{\sigma_1\circ\tau(y)}^2+\PAR{\sigma_2\circ\tau(y)}^2}+\nu_c(\by)\\
\nu_2(\by)&=
\frac{1}{2}\PAR{b^1\circ\tau(\by)+b^2\circ\tau(\by)}=b_1\PAR{-\frac{\ABS{y_1}}{\sqrt{2}}+y_2}+b_2\PAR{\frac{\ABS{y_1}}{\sqrt{2}}+y_2}
\end{align*}
where $\nu_c$ is a continuous function on $\dR^2$, with $\lim_{y_1\to 0}\nu_c(\by)=0$.

By assumption, $\nu_1$ is continuous on $\dR^2\setminus\BRA{y_1=0}$ and $\nu_2$ is continuous on $\dR^2$. Moreover, 
\begin{align*}
\lim_{y_1\to 0^-}\nu_1(\by)&=\frac{1}{\sqrt{2}}\PAR{b_1(y_2)-b_2(y_2)}-\alpha(y_2)\PAR{\PAR{\sigma_1(y_2)}^2+\PAR{\sigma_2(y_2)}^2}=0\\
\lim_{y_1\to 0^+}\nu_1(\by)&=\frac{1}{\sqrt{2}}\PAR{b_2(y_2)-b_1(y_2)}+\alpha(y_2)\PAR{\PAR{\sigma_1(y_2)}^2+\PAR{\sigma_2(y_2)}^2}=0,
\end{align*}
by Definition \eqref{eq:def-alpha} of $\alpha$. Then $\nu$ is continuous on $\dR^2$, which concludes the proof of this step.

\begin{quote}
    \underline{Step 3}: $\bZ$ satisfies an SDE with locally Lipschitz-continuous coefficients on $\dR^2$.
\end{quote}

The proof is almost identical to the proof of \cite[Theorem 3.20]{Leobacher-Szolgyenyi2017} except we have not assumed that $b$ is bounded. Consequently, the coefficients of the SDE related to $\bZ$ are locally Lipschitz-continuous functions on $\dR^2$, and not Lipschitz-continuous (even when $b$ is Lipschitz-continuous).
We are presenting this proof nonetheless for the sake of exhaustiveness and because it is fairly short. 

\medskip 
As Itô's formula holds for $G$, the process $\bZ=G(\bX)$  satisfies an SDE
\begin{equation}\label{eq:Z}
\dd \bZ_t=\tilde b(\bZ_t)\dd t+\tilde \sigma(\bZ_t)\dd W_t,
\end{equation}
where the drift term $\tilde b=\begin{pmatrix}\tilde b_1\\\tilde b_2\end{pmatrix}$ 
is given by, $k\in\BRA{1,2}$,
\begin{equation}\label{eq:b-tilde}
\tilde b_k(\bz)=G'_k(G^{-1}(\bz))b(G^{-1}(\bz))+\frac{1}{2}Tr\PAR{\sigma^\top(G^{-1}(\bz))G_k''(G^{-1}(\bz))\sigma(G^{-1}(\bz))}
\end{equation}
where $G'_k=\begin{pmatrix}\frac{\partial G_{k}}{\partial x_1}&\frac{\partial G_k}{\partial x_2} \end{pmatrix}$ is the   Jacobian matrix and 
$G''_k=\begin{pmatrix}\frac{\partial^2 G_{k}}{\partial x_1^2}&\frac{\partial^2 G_k}{\partial x_1x_2}
 \\\frac{\partial^2 G_{k}}{\partial x_1x_2}&\frac{\partial^2 G_k}{\partial x_2^2}
\end{pmatrix}$ is the Hessian matrix of the $k^\text{th}$-coordinate $G_k$ of the function $G$,
and the diffusion term is given  by
\begin{equation}\label{eq:sigma-tilde}
\tilde \sigma(\bz)=
G'(G^{-1}(\bz))\sigma(G^{-1}(\bz)).
\end{equation}

From the previous step, we know that $\tilde b$ is a continuous function on $\dR^2$. We also know that $G(\bx)=\bx$ for $x\notin \Theta_c$. Then $\tilde b(\bz)=b(\bz)$ for $z\notin \Theta_c$. By assumption on $b$, we deduce that $\tilde b$ is intrinsic Lipschitz on $\dR^2\setminus\Theta_c$.

We now study $\tilde b$ on $\Theta_c\setminus\Theta$. The function $G'=\begin{pmatrix}G'_1\\G'_2\end{pmatrix}$, given by \eqref{eq:G'1}-\eqref{eq:G'2}. The function $G^{-1}$ is Lipschitz-continuous on $\dR^2$, $G'$ is bounded and intrinsic Lipschitz-continuous on $\Theta_c\setminus\Theta$ ($G'$ is differentiable on $\Theta_c\setminus\Theta$ with bounded derivatives, see \cite[Lemma 3.8]{Leobacher-Szolgyenyi2017}), and $b$ is locally bounded by Assumption $(A_\PAR{b,\sigma})$, consequently we deduce that $z\mapsto G'_k(G^{-1}(\bz))b(G^{-1}(\bz))$  is locally intrinsic Lipschitz-continuous on $\Theta_c\setminus\Theta$ (see \cite[Lemma 3.9]{Leobacher-Szolgyenyi2017}).

Similarly, $G''$ is differentiable on $\Theta_c\setminus\Theta$ with bounded derivatives, and consequently intrinsic Lipschitz-continuous on $\Theta_c\setminus\Theta$. As $G''$ is bounded and $\sigma$ locally Lipschitz-continuous on $\dR^2$, we also conclude that the second term of $\tilde b$, in \eqref{eq:b-tilde}, is  locally Lipschitz-continuous on $\dR^2$. By continuity of $\tilde b$, we conclude that $\tilde b$ is locally Lipschitz-continuous \cite[Lemmas 3.6 and 3.11]{Leobacher-Szolgyenyi2017}.

As $\sigma$ is locally Lipschitz-continuous on $\dR^2$, we also easily conclude
that $\tilde\sigma$, given by \eqref{eq:sigma-tilde},  is locally Lipschitz-continuous on $\dR^2$.
\end{proof}

\subsection{Proof of Theorem~\ref{thm:main result}}\label{sec:proof main thm}

We are now ready to present the proof of Theorem~\ref{thm:main result}.

\begin{proof}[Proof of Theorem~\ref{thm:main result}] 
By Theorem \ref{thm:SDE-Z}, $Z$ is solution of SDE with Locally Lipschitz-coefficients. We deduce by  \cite[Theorem 1.1.8]{Hsu2002} that there is strong existence and uniqueness of the process $Z$ with initial condition $\bZ_0:=G(\bX_0)$, up to a potential explosion time $e(Z)$. Recall that $G$ is a $\cC^1$-diffeomorphism by Theorem \ref{prop:G-diffeo}. Since Itô's formula holds for $G^{-1}$ (see Step 1 in the proof of Theorem \ref{thm:SDE-Z}) and $G^{-1}(\bZ)$ is solution of \eqref{eq:EDSX}, the conclusion follows.
\end{proof}

\section{Rank-based interacting systems in higher dimensions}\label{sec: weak result}

In this section, we  first prove the weak well-posedness of the systems \eqref{eq:EDS-poisson} and \eqref{eq:EDS-poisson-2} in any dimension $N\geq 2$, under good assumptions, including the uniform ellipticity and weak conditions on the drift terms. Then we study the positivity of the solution when the diffusion coefficients vanish at zero.

In this section, we denote the set $\BRA{1,\ldots, N}$ by $\DSBRA{1,N}$.

\subsection{Weak well-posedness of the SDE~\eqref{eq:EDS-poisson} with positive and constant diffusion coefficients}\label{sec:weak existence 1}

 We consider the following specific case of the rank-based interacting diffusion model~\eqref{eq:EDS-poisson}, defined by
\begin{align}\label{eq:gene-atlas}
    \dd X^{i}_t&=\sum_{k=1}^Nb_k(X^{i}_t)\ind_{X^{i}_t=X^{(k)}_t}\dd t
    +\sum_{k=1}^N\sigma_k\ind_{X^{i}_t=X^{(k)}_t}\dd W^i_t
\end{align}
where $\PAR{W^i}_{1\leq i\leq N}$ are $N$ independent Brownian motions, and $X^{(1)},X^{(2)},\ldots,X^{(N)}$ are the order statistics of the N-tuple $\bX=(X^{1},X^{2},\ldots,X^{N})$. We first notice that the system \eqref{eq:gene-atlas} is well defined because the $X^i$ are distinct almost surely.

Using the notations of Bass and Pardoux, \cite{BP87}, Equation \eqref{eq:gene-atlas} can be written 
\begin{align*}
     \dd X^{i}_t=\sum_{k=1}^Nb_k(X^{i}_t)\ind_{\cQ^{i}_{(k)}}\PAR{\bX_t}\dd t
    +\sum_{k=1}^N\sigma_k\ind_{\cQ^{i}_{(k)}}\PAR{\bX_t}\dd W^i_t
\end{align*}

where $\bx=(x^1,\ldots,x^N)\in \cQ^{i}_{(k)}$ means that $x^i$ is ranked $k$th among $x^1,\ldots, x^N$:
\begin{align*}
   \cQ^{i}_{(1)}&=\BRA{\bx\in\dR^N: x^i=\min_{j\in\BRA{1,\ldots,N}}x^j } \\
   \cQ^{i}_{(N)}&=\BRA{\bx\in\dR^N: x^i=\max_{j\in\BRA{1,\ldots,N}}x^j }\\
    \cQ^{i}_{(k+1)}&=\BRA{\bx\in\dR^N: x^i>\max_{1\leq r\leq k}x^{j_r}\text{ for some }j_1,\ldots, j_k,\text{ and }x^i= \min_{j\notin\BRA{j_1,\ldots, j_k}}x^j }.
\end{align*}
We easily observe that for any fixed $k\in\DSBRA{1,N}$, $\PAR{\cQ^{i}_{(k)}}_{1\leq i\leq N}$ is a partition of $\dR^N$, and for any fixed $i\in\DSBRA{1,N}$, $\PAR{\cQ^{i}_{(k)}}_{1\leq k\leq N}$ is also a partition of $\dR^N$.

The infinitesimal generator of the system \eqref{eq:gene-atlas} is 
\begin{equation}\label{eq:generator}
\cA f(\bx)={1\over 2}\sum_{i,j=1}^N A_{ij}(\bx){\partial^2f(\bx) \over \partial x_i\partial x_j}+\sum_{i=1}^N B_i(\bx){\partial f(\bx)\over \partial x_i}
\end{equation}
where 
\begin{itemize}
\item $A=\PAR{A_{ij}}_{1\leq i,j\leq N}=\eta^t\eta$ with $\eta=\PAR{\eta_1,\cdots,\eta_N}$ and $\eta_i(\bx)=\sum_{k=1}^N\sigma^N_k\ind_{\cQ^i_{(k)}}(\bx)$,
\item $B_i(\bx)=\sum_{k=1}^Nb_k(x^i)\ind_{\cQ^i_{(k)}}(\bx)$.
\end{itemize}

\begin{prop}\label{prop:wellposed}
We assume that $\sigma_k^N>0$ and that
 $b_k^N:\dR\to \dR$ are locally bounded measurable functions with at most linear growth at infinity  for all $k\in\DSBRA{1,N}$. Let $\bx_0\in\dR^N$.
 Then there exists a solution to the martingale problem for $\cL$ starting at $\bx_0$ and that solution is unique.
\end{prop}

\begin{proof}
The diffusion coefficients being positive, we easily observe that $A$ is a bounded and uniformly positive definite measurable function: $\exists \lambda>0$: $\forall \bx,\by\in\dR^N$, $\by^tA(\bx)\by\geq \lambda \NRM{y}^2$.

As $B$ is locally bounded, thanks to \cite[Theorems 6.4.3 and 10.2.2]{SV06}, it suffices to consider the case where $B\equiv 0$. 
The existence of a solution to the martingale problem with $B\equiv 0$ is given by \cite[Exercise 12.4.3]{SV06}.

As the matrix $A$ satisfies the assumptions of Bass and Pardoux, we deduce by \cite[Theorem 2.1]{BP87} the martingale problem for $\cL$ starting at $\bx_0$ is well-posed.
\end{proof}

\subsection{Weak well-posedness of the  SDE~\eqref{eq:EDS-poisson-2}}\label{sec:weak existence 2}

Using the notation of the previous section,
Equation \eqref{eq:EDS-poisson-2} can be written in the following way 
\begin{align*}
     \dd X^{i}_t=\sum_{k=1}^N b_k(X^{i}_t)\ind_{\cQ^{i}_{(k)}}\PAR{\bX_t}\dd t
    +\sigma_i(X^{i}_t)\dd W^i_t.
\end{align*}
and then its infinitesimal generator is given by
\begin{equation}\label{eq:generator}
\cB f(\bx)={1\over 2}\sum_{i,j=1}^N A_{i,j}(\bx){\partial^2f(\bx) \over \partial x_i\partial x_j}+\sum_{i=1}^N B_i(\bx){\partial f(\bx)\over \partial x_i}
\end{equation}
with $A_{ij}(\bx)=0$ for $i\neq j$ and $A_{ii}(\bx)=\sigma_i^2(x_i)$, and $B_i(\bx)=\sum_{k=1}^Nb_k(x^i)\ind_{\cQ^i_{(k)}}(\bx)$ 

\begin{prop}\label{prop:wellposed-2}
Let $\bx_0\in\dR^N$. We assume that $A$ is positive  and $\sigma_i$ is Hölder-continuous of order $\eta\geq \frac{1}{2}$ on $\dR$, for all $ i\in\DSBRA{1, N}$. 
We also assume that  $b_i:\dR\to \dR$ is a locally bounded measurable function, with at most linear growth at infinity, for all $i\in\DSBRA{1,N}$.

Then, there exists a solution to the martingale problem for $\cB$ starting at $\bx_0$ and that solution is unique.
\end{prop}

\begin{proof}
  
    The assumptions of Theorem~\ref{prop:wellposed-2} imply that
    \begin{itemize}
        \item $B: \dR^N\to \dR^N$ is a measurable function, with at most linear growth  at infinity: there exists $K>0$ such that $\ds{\bx^tB(\bx)\leq K(1+\ABS{\bx}^2)}$;
        \item $A:\dR^N\to \dR^N\times\dR^N$ is a continuous function with values in the set of symmetric positive definite matrices, with at most linear growth at infinity: $\exists \lambda>0$: $\forall \bx,\by\in\dR^N$, $\by^tA(\bx)\by\geq \lambda \NRM{y}^2$.
    \end{itemize}
     Let $A_n$ and $B_n$ bounded functions satisfying the same assumptions as $A$ and $B$, such that $A_n=A(x)$ and $B_n(x)=B(x)$ on $\cB_n=\BRA{x\in\dR^N:\NRM{x}\leq n}$. By \cite[Theorem 6.4.3]{SV06}, well-posedness of the martingale problem with coefficients $(A_n,B_n)$ is equivalent to well-posedness of the martingale problem with coefficients $(A_n,0)$. 
     As $A$ is continuous, we note that $A_n$ is bounded. Then the existence of a solution to the martingale problem is given by \cite[Exercise 12.4.3]{SV06}.\\ 
     By \cite[Theorem 1]{YW71}, when the functions $\sigma_i$ are a $\alpha$-Hölderian function, with $\alpha\geq 1/2$, there is uniqueness to the martingale problem $(A_n,0)$.
     \\
     Finally,  by \cite[Theorem 10.2.2]{SV06}, we deduce that the martingale problem with coefficient $(A,B)$ is well-posed.
\end{proof}

\begin{rk}  As we can see in \cite[Theorem 6.4.2]{SV06}, the positivity of the matrix $A$ is crucial in the proof to use \cite[Theorem 6.4.3]{SV06}.
\end{rk}

\subsection{Non-negativity of the solution}\label{App:nonnegativity}

Let $N\geq 1$. 
In this section, we consider the interacting model with drift and diffusion coefficients depending on the rank defined by 
\begin{align}\label{eq:EDS-poisson1}
    \dd X^{i}_t&=\sum_{k=1}^N b( X^{i}_t)\ind_{X^{i}_t=X^{(k)}_t}\dd t
    +\sum_{k=1}^N \sigma_k (\max( X^{i}_t,0))\ind_{X^{i}_t=X^{(k)}_t}\dd W^i_t,
\end{align}
and the rank-based interacting diffusion model with only drift coefficients depending on the rank defined by
\begin{align}\label{eq:EDS-poisson2}
    \dd X^{i}_t&=\sum_{k=1}^N b( X^{i}_t)\ind_{X^{i}_t=X^{(k)}_t}\dd t
    +\sigma_i (\max( X^{i}_t,0))\dd W^i_t,
\end{align}
where $\bX_0=(X_0^{1},\ldots, X_0^{N})$ with $0<X_0^{1}<X_0^{2}<\ldots<X_0^{N}$,  
$\PAR{W^i}_{1\leq i\leq N}$ are $N$ independent Brownian motions, and $X^{(1)},X^{(2)},\ldots,X^{(N)}$ are the order statistics of the N-tuple $\bX=(X^{1},X^{2},\ldots,X^{N})$.

For a vector $\bx=(x^1,\ldots,x^N)\in\dR^N$, we use the notation $\bx\geq 0$ when $x^i\geq 0$ for all  $i\in\DSBRA{1,N}$.

\begin{prop}
We assume $b_i $ are continuous functions on $\dR$ and are ordered at $zero$: $0\leq  b_1(0)\leq b_2(0)\leq \cdots\leq b_N(0)$, and $\sigma_i$ are continuous functions such that $\sigma_i (0)=0$ and $\forall x>0$ $\sigma_i(x)>0$. We assume that there is a unique strong solution $\bX$ to \eqref{eq:EDS-poisson1} (respectively \eqref{eq:EDS-poisson2}) with an initial condition satisfying $0<X_0^{1}<X_0^{2}<\ldots<X_0^{N}$ a.s. Then the solution to \eqref{eq:EDS-poisson1} (respectively \eqref{eq:EDS-poisson2}) satisfies  for all $t\geq0$, $\bX_t\geq 0$ a.s. 
\end{prop}

\begin{rk}
Instead of assuming that the functions $b_i$ are ordered at zero, we may alternatively consider that either all $b_i(0)>0$, or that $b_1(0)=0$, as we can see in the proof below.
\end{rk}

\begin{proof}

We only provide proof for \eqref{eq:EDS-poisson1}, since the arguments are exactly the same for \eqref{eq:EDS-poisson2}.

The proof is inspired by the one in \cite[Chap. IV - Example 8.2]{IW89}.
\begin{quote}
  \underline{First case} : 
 We assume that $\forall i\in\DSBRA{1,N}$, $b_i (0)> 0$.

\end{quote}

By continuity of the functions, let $\varepsilon>0$ be such that for all $i\in \DSBRA{1,N}$,  $\sup_{x\in[-\varepsilon, \varepsilon]} b_i (x)>0$

We introduce $\tau_{\varepsilon}=\inf\BRA{t\geq 0: \exists i\in\DSBRA{1,N} : X^{i}_t =-\varepsilon}$ and $J_\varepsilon=\BRA{i\in\DSBRA{1,N}: X^{i}_{\tau_\epsilon} =-\epsilon }$ (with the convention $\inf\emptyset=+\infty$).

We now prove by contradiction that $\tau_\varepsilon=+\infty$ a.s.. Assume that $\dP(\tau_{\varepsilon}<\infty)>0$. Let $j\in J_\varepsilon$ such that
$X^{j}_{\tau_{\varepsilon}} =-\varepsilon $. 
We take any $r<\tau_{\varepsilon}  $ such that  $X^{j}_t<0$ for all $j\in J_\varepsilon$ and $t\in(r,\tau_{\varepsilon})$. Such $r$ exists because there is at most $N$ such $j$. Then with probability one, we have
\[
\rd X^{j}_t =\sum_{k=1}^N b_k( X^{j}_t)\ind_{X^{j}_t=X^{(k)}_t}\dd t
\] 
on the interval $(r,\tau_{\varepsilon})$. Hence, by the choice of $\varepsilon$, for $j\in J_\varepsilon$, $t\mapsto X^{j}_t$ is increasing on this interval. This is clearly a contradiction with $X^{j}_{\tau_{\varepsilon}} =-\varepsilon $.  
Consequently, $\tau_\varepsilon=+\infty$ a.s., and thus $\bX_t>-\varepsilon$ a.s. for all $\varepsilon>0$ small enough. Taking a sequence $(\varepsilon_n)\in\mathbb{Q}^{\dN}$ converging to $0$, we deduce that $\bX_t\geq 0$ a.s. for $\forall t\geq 0$.

\begin{quote}
    \underline{Second case} :
There exists $i$ such that  $b_i(0)=0$.  
\end{quote}

Given that the functions $b_i$ are ordered at 0 and that at least one of them vanishes at 0, we can conclude that $b_1(0)=0$. 

We proceed by induction on the dimension $N$, and the process $\bX$ is denoted $\bX^N=\PAR{X^{1,N},\ldots, X^{N,N}}\in\dR^N$ in the sequel.
The initialization at $N=1$  is immediate using the same arguments as in \cite[Chap. IV - Example 8.2]{IW89}.

Assume that the statement holds for $N-1>0$; we now show that it also holds for $N$.
Let us note that $\bX^N_t>0$ for all $t\in[0,\tau)$ with $\tau =\inf\BRA{t\geq 0: \exists i\in\DSBRA{1,N} \, X^{i,N}_t =0}$.

On $\BRA{\tau=\infty}$, the result holds.
On $\BRA{\tau<\infty}$, we introduce $I=\BRA{i\in\DSBRA{1,N}: X^{i,N}_{\tau} =0}$.
There exists an $i_1\in I$ on $\BRA{\tau<\infty}$
associated with the drift coefficient $b_1$ and the volatility coefficient $\sigma_1$.

By the strong Markov property, on $\BRA{\tau<+\infty}$,
the process $\widetilde{\bX}^N_{t}:=\bX^N_{t+\tau}$ is a solution of (\ref{eq:EDS-poisson1}) starting from $\widetilde{\bX}^N_{0} = \bX^N_{\tau}$. 

Hence, as $b_{1} (0)=0$ and $\sigma_1(0)=0$, by strong uniqueness of the particle system, $\widetilde{X}^{i_1,N}\equiv 0$ a.s.   on $\BRA{\tau<\infty}$.

We now introduce the process $\widetilde\bX^{N-1}\in\dR^{N-1}$ obtained by removing the $i_1$-th coordinate of $\widetilde\bX^N$. The associated coefficients of the SDE satisfied by $\widetilde\bX^{N-1}$ are $b_2,\ldots,b_N$ and $\sigma_2,\ldots,\sigma_N$. 

If $b_2(0)>0$, by assumption it implies that $b_i(0)>0$ for all $i\in\DSBRA{2,N}$, then we can use the \emph{First case} of this proof to study the sign of $\widetilde\bX^{N-1}$, and the conclusion follows directly.

Otherwise, since the vector $\widetilde\bX^{N-1}$ is of dimension $N -1$, we conclude by induction for all $i\in\DSBRA{1,N}$ with $i\neq i_1$, that $\widetilde X^{i,N}_t\ge 0$ a.s.
Therefore, a.s. for all $i\in\DSBRA{1,N}$  and $t\geq 0$, $X^{i,N}_t\ge 0$, by uniqueness of the particle system.

\end{proof}

\noindent \textbf{Acknowledgements:} We would like to thank Fabrice Mahé for introducing us to deterministic models related to rainbow trout growth. We also thank Tomoyuki Ichiba for the fruitful discussions on the Atlas model in Cambridge. \\
HG would like to thank the Isaac Newton Institute for Mathematical Sciences, Cambridge, for its support and hospitality during the programme \emph{Stochastic systems for anomalous diffusion}.
HG acknowledges funding from the Natural Sciences and Engineering Research Council of Canada (NSERC) through its Discovery Grant (RGPIN-2020-07239).\\
This work was supported by the Research Institute for Mathematical Sciences,
an International Joint Usage/Research Center located at Kyoto University.

\appendix
\renewcommand{\theequation}{\Alph{section}.\arabic{equation}}
\let \sappend=\section
\renewcommand{\section}{\setcounter{equation}{0}\sappend}

\phantomsection
\addcontentsline{toc}{section}{\refname}%
\bibliographystyle{alpha}
\bibliography{biblio-poisson}

\end{document}